\documentclass[english]{amsart}
\usepackage{ucs}
\usepackage[utf8x]{inputenc}
\usepackage{babel}
\usepackage{amsmath,amsthm}
\usepackage{amssymb}
\usepackage[all]{xy}
\usepackage{verbatim}
\newcommand{\C}{\mathbb{C}}

\newcommand{\Z}{\mathbb{Z}}

\newcommand{\ra}{\rightarrow}

\newtheorem{thm}{Theorem}

\newtheorem{prop}[thm]{Proposition}

\newtheorem{cor}[thm]{Corollary}

\title{A remark on the Torelli theorem for cubic fourfolds}

\author{Fran\c{c}ois Charles}
\email{francois.charles@univ-rennes1.fr}
\address{Universit\'e de Rennes 1, IRMAR--UMR 6625 du CNRS, Campus de Beaulieu, 35042 Rennes
Cedex, France}

\begin{document}

\begin{abstract}
In this note, we give a short proof of the Torelli theorem for cubic fourfolds that relies on the global Torelli theorem for irreducible holomorphic symplectic varieties proved by Verbitsky.
\end{abstract}

\maketitle

\section{Introduction}

A Torelli theorem for a given family of complex projective varieties states, roughly, that two members of this family can be distinguished by Hodge-theoretic data. In this paper, we are interested with smooth cubic hypersurfaces in $\mathbb P^5_{\C}$. 

The interesting part of the cohomology algebra of a cubic fourfold $X$ lies in degree $4$. The cohomology class $h$ of a hyperplane section induces a distinguished class $h^2\in H^4(X, \Z)$. The global Torelli theorem for cubic fourfolds was proved by Voisin in \cite{VoisinT} and reproved recently by Looijenga in \cite{LooijengaT} using different methods.

\begin{thm}[\cite{VoisinT, LooijengaT}]\label{Torellicubic}
Let $X$ and $X'$ be two smooth complex cubic fourfolds, and let 
$$\phi : H^4(X, \Z)\ra H^4(X', \Z)$$
be an isomorphism of polarized Hodge structures preserving the class $h^2$ of a linear section. Then there exists a projective isomorphism $f : X'\ra X$ such that $\phi=f^*$.
\end{thm}

In both the proofs of Voisin and Looijenga, Theorem \ref{Torellicubic} is proved, among other considerations, through a detailed analysis of the geometry of a distinguished class of cubic fourfolds, be it the ones containing a plane in \cite{VoisinT} or specific singular ones in \cite{LooijengaT}. 

In the recent paper \cite{VerbitskyT}, Verbitsky has proven a global Torelli theorem for irreducible holomorphic symplectic varieties. We refer to the nice surveys \cite{MarkmanT} and \cite{HuybrechtsT} for a thorough discussion of this result. Unlike the proofs of the Torelli theorem for cubic fourfolds, Verbitsky's proof does not rely on the study of specific irreducible holomorphic symplectic varieties, focusing instead on making full use of the hypercomplex structure of these varieties.

As is well-known from the work of Beauville-Donagi \cite{BD} -- and as used in Voisin's work in \cite{VoisinT} -- the Fano variety of lines of a cubic fourfold is an irreducible holomorphic symplectic variety. In this note, we make use of this fact to show that Theorem \ref{Torellicubic} can be deduced in an elementary way from Verbitsky's global Torelli theorem. As the discussion above shows, this leads to a proof of Theorem \ref{Torellicubic} that does not rely in any way on the study of specific cubic hypersurfaces.

\textbf{Acknowledgements.} I was happy to benefit from helpful discussions with Claire Voisin and with Arnaud Beauville, who also pointed me to the reference \cite{AltmanKleiman}. It is a pleasure to thank them.

\section{The period map and Verbitsky's theorem}

Let $X$ be a cubic fourfold over $\C$. Its fourth singular cohomology group $H^4(X, \Z)$ with integral coefficients is endowed with a polarized Hodge structure and contains a distinguished class $h^2$ that is the square of a hyperplane section. Let $H^4(X, \Z)_0$ denote the primitive cohomology group of $X$, that is, the orthogonal of $h^2$ with respect to cup-product.

Let $(L,u)$ be an abstract lattice with a distinguished element $u$ isomorphic to the pair $(H^4(X, \Z), h^2)$. Let $L_0$ be the orthogonal complement of $u$ in $L$ and let $\mathcal D$ be the period domain associated to $L$, that is, 
$$\mathcal D=\{x\in \mathbb P(L_0\otimes \C), (x, x)=0\, \mathrm{and}\, (x, \overline x)>0\}.$$
Let $\mathcal M$ be the moduli space of marked smooth cubic fourfolds. It parametrizes pairs $(X, \phi)$ where $X$ is a smooth cubic fourfold and $\phi : H^4(X, Z)\ra L$ is an isometry sending $h^2$ to $u$, up to projective isomorphism. The period map $p : \mathcal M \ra \mathcal D$ sends a pair $(X, \phi)$ to the line $\phi(H^{3, 1}(X))$ in $\mathcal D$. Standard Hodge theory shows that $p$ is a local isomorphism. The following follows formally from the definitions.

\begin{prop}\label{enough}
Theorem \ref{Torellicubic} holds if and only if $p$ is injective.
\end{prop}

Let $F$ be the variety of lines on $X$. By \cite{BD}, $F$ is an irreducible holomorphic symplectic variety. Via the Pl\"ucker embedding, $F$ is endowed with a canonical polarization. By an abuse of notation, we denote by $h\in H^2((F, \Z)$ the class of the corresponding line bundle.

By the general theory of \cite{Beauville}, the second cohomology group of $F$ $H^2(F, \Z)$ is endowed with a canonical polarization $q$, the Beauville-Bogomolov form. By \cite{BD}, the Abel-Jacobi map induces an isometry between the primitive cohomology groups
\begin{equation}\label{AJ}
i : H^4(X, Z)_0\ra H^2(F, \Z)_0.
\end{equation}
Let $(L',u')$ be an abstract lattice with a distinguished element $u'$ isomorphic to the pair $(H^2(F, \Z), h)$. We identify the orthogonal complement of $u'$ in $L'$ to $L_0$.

Let $\mathcal M'$ be the connected component of the moduli space of marked irreducible holomorphic symplectic varieties $(F, \psi :  (H^2(F, \Z), h)\ra (L', u'))$ containing the Fano varieties of lines of cubic fourfolds. As above, let $p' : \mathcal M'\ra \mathcal D$ be the period map that sends $(F, \psi)$ to the line $\psi(H^{2, 0}(F))$ in $\mathcal D$. By \cite[Th\'eor\`eme 5]{Beauville}, $p'$ is a local isomorphism.

Finally, let $\mathcal N$ be the incidence variety that parametrizes tuples $(X, F, \phi, \psi)$ where $X$ is a cubic fourfold, $F$ its variety of lines, and $\phi$, $\psi$ are marking of $X$ and $\psi$ respectively such that $\phi_{|H^4(X, \Z)_0}=\psi\circ i$. Let $\pi$ and $\pi'$ denote the two canonical projections from $\mathcal N$ to $\mathcal M $ and $\mathcal M'$ respectively. An elementary computation shows that $\pi$ and $\pi'$ are both local isomorphisms. Furthermore, $\pi$ is injective.

The maps defined above fit into the following commutative diagram.

\begin{equation}
\xymatrix{
& \mathcal N\ar[dl]^{\pi}\ar[dr]^{\pi'}\\
\mathcal M \ar[dr]^p&  & \mathcal M'\ar[dl]^{p'}\\
& \mathcal D}
\end{equation}

The result we state now, due to Markman in \cite{MarkmanT}, combines Verbitsky's global Torelli theorem in \cite{VerbitskyT} with a result of Huybrechts \cite{Huy} describing non-separated points in the moduli space of holomorphic symplectic varieties. It is the key input of this note.

\begin{prop}
The period map $p'$ is generically injective on each connected component of $\mathcal M'$.
\end{prop}

\begin{proof}
If $\C u$ be a very general point of $\mathcal D$, there is no element $x\in L$ orthogonal to $u$. By \cite[Theorem 2.2, (5)]{MarkmanT}, which encompasses both the global Torelli theorem of Verbitsky and Huybrechts' result alluded above, the fiber of $p'$ at $\C u$ contains exactly one point in each connected component of $\mathcal M'$.
\end{proof}

\section{Recovering a cubic hypersurface from its variety of lines, and proof of Theorem \ref{Torellicubic}}

Before proving Theorem \ref{Torellicubic}, we show that the variety of lines of a cubic determines the cubic itself.

\begin{prop}
Let $k$ be a field of characteristic different from $3$. Let $X$ and $X'$ be two cubic hypersurfaces of dimension $d\geq 3$ over $k$ with isolated singularities. Let $F$ and $F'$ be the Fano varieties of lines on $X$ and $X'$ respectively.

Let $g : F\ra F'$ be a projective isomorphism with respect to the Pl\"ucker embeddding of $F$ and $F'$. Then there exists a projective isomorphism $f : X\ra X'$ inducing $g$.
\end{prop}

\begin{proof}
Let $V$ be the standard $(d+2)$-dimensional vector space over $k$, so that $X$ and $X'$ are hypersurfaces in $\mathbb P(V)$. Let $G$ be the grassmannian variety of lines in $\mathbb P(V)$. The Pl\"ucker embedding is the canonical embedding of $G$ in $\mathbb P(\bigwedge^2 V)$. 

The Fano varieties $F$ and $F'$ are subvarieties of $G$. By \cite[1.16 (iii)]{AltmanKleiman}, and since $G$ is defined by quadratic equations in $\mathbb P(\bigwedge^2 V)$, the intersection of the quadric hypersurfaces in $\mathbb P(\bigwedge^2 V)$ containing $F$ (resp. $F'$) is equal to $G$. It follows from this remark that there exists a projective automorphism $g'$ of $G$ sending $F$ to $F'$ and such that $g'_{|F}=g$.

By a classical theorem of Chow \cite{Chow}, there exists an automorphism $f'$ of $\mathbb P(V)$ inducing $g'$.This means that for any line $L$ on $X$ corresponding to a point $\ell$ of $F$, the line $f'(L)$ lies on $X'$ and corresponds to the point $g(\ell)\in F'$. It follows easily that $f'$ sends $X$ to $X'$ and that the restriction of $f'$ to $X$ is a projective isomorphism $f : X\ra X'$ inducing $g$.
\end{proof}

The next corollary is immediate.

\begin{cor}
The map $\pi'$ is injective.
\end{cor}

\begin{prop}\label{comp}
The period map $p$ is injective on each connected component of $\mathcal M$.
\end{prop}

\begin{proof}
Let $O(L, u)$ be the group of orthogonal automorphisms of $L$ fixing $u$. The group $O(L, u)$ acts on both $\mathcal M$ and $\mathcal D$, and the map $p: \mathcal M\ra \mathcal D$ is equivariant with respect to this action. In particular, we only have to find a connected component $\mathcal M^0$ of $\mathcal M$ such that the restriction of $p$ to $\mathcal M^0$ is injective.

Let $\mathcal N^0$ be a connected component of $\mathcal N$, and let $\mathcal M^0$ (reps. $\mathcal M'^0$) be the closure of the image of $\mathcal N^0$ in $\mathcal M$ (resp. $\mathcal M'$) by $\pi$ (resp. $\pi'$). In the diagram 
\begin{equation}
\xymatrix{
& \mathcal N^0\ar[dl]^{\pi}\ar[dr]^{\pi'}\\
\mathcal M^0 \ar[dr]^p&  & \mathcal M'^0\ar[dl]^{p'}\\
& \mathcal D}
\end{equation}
the maps $\pi'$ and $p'_{|\mathcal M'^0}$ are generically injective, and the map $\pi_{|\mathcal N^0}$ is dominant, which proves that $p'_{|\mathcal M^0}$ is generically injective, hence injective since $\mathcal M^0$ is separated.
\end{proof}

\begin{proof}[Proof of Theorem \ref{Torellicubic}]
By Proposition \ref{enough}, we only have to prove that the period map $p$ is injective. Let $O^+(L, u)$ be the subgroup of $O(L, u)$ consisting of elements of real spinor norm $1$. The two components of the period domain $\mathcal D$ are exchanged under the action of $O(L, u)$, and the subgroup that fixes the components is $O^+(L, u)$.

On the moduli side, it is a result of Beauville in \cite[Th\'eor\`eme 2]{beauvillemon} that $\mathcal M$ has two components which are exchanged by $O(L, u)$, and that the subgroup that fixes the components is $O^+(L, u)$ as well. Since $p$ is injective on each connected component of $\mathcal M$, this proves that $p$ is injective.

\end{proof}

\noindent\textbf{Remark.} It is possible to investigate variants of this situation. For instance, in \cite{VoisinT}, the period map for marked cubic fourfolds containing a plane is thoroughly investigated, and it is proven there that the period map is $2$ to $1$. This discrepancy from Theorem \ref{Torellicubic} is exactly due to the failure of the monodromy argument above, namely, to the fact that the group fixing the connected components of the moduli space of marked cubic fourfolds containing a plane is smaller than the group fixing the components of the corresponding period domain.

\bibliography{torelli}{}
\bibliographystyle{plain}

\end{document}